\documentclass[11pt]{amsart}
\usepackage[parfill]{parskip}    
\usepackage{fullpage}
\usepackage{graphicx}
\usepackage{amssymb,amsmath,latexsym}
\usepackage{tikz}
\usetikzlibrary{arrows,decorations.markings,decorations.pathreplacing,calc,matrix,intersections}
\tikzset{->-/.style={decoration={markings,mark=at position #1 with {\arrow{>}}},postaction={decorate}}}

\usepackage{epstopdf}
\usepackage{ifpdf}
\usepackage{color}
\usepackage{float}
\usepackage{amscd,amsmath, mathabx}
\usepackage{amssymb,amsmath,latexsym,color,enumerate,tikz}
\usepackage[all]{xypic}       
\usepackage[varg]{pxfonts}
\usepackage{amsmath,amsthm,amsfonts,amssymb,fancyhdr,graphics,relsize,tikz-cd,mathtools,faktor,tikz,changepage}
\usepackage{graphicx}
\usepackage{placeins}

\usepackage{dsfont}
\definecolor{red}{rgb}{1,0,0} 

 \definecolor{darkgreen}{rgb}{0, .7, 0}

 \definecolor{purple}{rgb}{.7, 0, 1}

\tikzset{mynode/.style={draw,circle,fill=black,inner sep=2pt,outer sep=0.5pt}}

\newtheorem{theorem}{Theorem}[section]
\newtheorem*{theorem*}{Theorem}
\newtheorem*{lemma*}{Lemma}
\newtheorem{proposition}[theorem]{Proposition}

\newtheorem{corollary}[theorem]{Corollary}

\theoremstyle{definition}

\theoremstyle{remark}

\usepackage{hyperref}
 
\begin{document}
\title{Self-similarity of the generalized Baumslag-Solitar groups}
\author{Dessislava H. Kochloukova}
\address{State University of Campinas (UNICAMP)}
\email{desi@unicamp.br}

\maketitle

\begin{abstract} We show that all residually finite generalized Baumslag-Solitar groups  of rank $n \geq 1$, defined on a finite and connected graph, are self-similar. Furthermore we prove that all residually finite fundamental groups of (finite, connected)  graph of groups where all vertex and edge groups are torsion-free and commensurable with the Heisenberg group and all edge groups properly embed in the corresponding vertex groups are self-similar.
\end{abstract}

\section{Introduction}

The classical Baumslag-Solitar groups were first defined in \cite{B-S} as an HNN-extension of an infinite cyclic group with non-trivial associated subgroups i.e. $BS(n,m) = \langle a,t \ | \ t a^n t^{-1} = a^m  \rangle$. This class of groups was generalised independently by Forester in \cite{For}  and  Whyte in \cite{Why} as fundamental group of a finite graph of groups with infinite cyclic vertex and edge groups and it was later studied by Levitt in  \cite{Lev1}, \cite{Lev2}.  This notion
was recently generalised by Button \cite{But} to the class of finite graphs of groups
with vertex and edge groups isomorphic to $\mathbb{Z}^n$  such that all edge groups have finite index in the corresponding vertex groups, he calls such groups generalised Baumslag–Solitar groups of rank $n \geq 1$. We will consider only generalised Baumslag-Solitar groups whose underlying graph is {\bf connected}. In \cite{B-I-W}  Bux, Isenrich, Wu considered    the
class of groups consisting of all (finite) graphs of groups with the property that all edge
and vertex groups are abstractly commensurable with $G$ and all edge group inclusions in
vertex groups have finite index. They called this class $BS_G$ and proved that if a group in this class acts faithfully on its Bass-Serre tree then it embeds in a finitely presented simple group.

  A self-similar group (or a state closed group) of degree $m$ is a group $G$ admitting a faithful, self-similar action on $m$-ary rooted tree. Among those groups are some well known examples as the Grigorchuk group \cite{G} and the Gupta-Sidki group \cite{Sidki}. 
Self-similar groups are always residually finite, but in general the converse does not hold. By a result of Hall in \cite{Hall} all finitely generated abelian-by-nilpotent groups are residually finite, in particular every finitely generated metabelian group is residually finite but as shown by Kochloukova and Souza in \cite{K-S} there exist finitely presented metabelian groups that are not self-similar.   Still in  \cite{K-Si} Kochloukova and Sidki  constructed   many finitely generated, metabelian, self-similar groups. In \cite{Ba-S} Bartholdi and Sunik proved that the Baumslag-Solitar group $BS(1,m)$ is self-similar for $m \not= -1$. Still we are faraway from understanding which finitely generated metabelian groups are self-similar.

We are interested which generalised Baumslag–Solitar groups of rank $n \geq 1$ or groups  from the more general class $BS_G$ are self-similar. As the groups from $BS_G$ are not necessary residually finite it is necessary first to understand which ones are residually finite. For example $BS(n,m)$ is residually finite only if $n =1$ or $m = 1$ or $n =  m$ or $-m$. The residual finiteness of the generalized Baumslag-Solitar groups of rank $n \geq 1$  was considered by Levitt in \cite{Lev2} and by Raptis, Talleli and Varsos in \cite{R-T-V}. 
 Recently  Lopez de Gamiz Zearra and Shepherd in \cite{L-S}   gave complete description of the generalised Baumslag-Solitar groups of rank $n$  that are residually finite, namely they are  ascending HNN extensions over  base group $\mathbb{Z}^n$ or are virtually $\mathbb{Z}^n$-by-free.

 Our approach to prove self-similarity is via the construction of virtual endomorphisms.
 In the case of transitive self-similar groups the connection between virtual endomorphism and self-similarity was suggested by Nekrashevych and Sidki  \cite{N-S}. The generalization of this result for (intransitive) self-similar groups was established by Dantas,  Santos and Sidki  \cite{Brazil}. 

As the fundamental groups of graphs of groups contain many free subgroups it is not surprising that in our proofs we use that any finitely generated free group is self-similar. This fact follows  from the specific construction of  an automata that generates free product  of  groups of order 2. The most general case was established by Savchuk and Vorobets  \cite{S-V}   but previously several results in this direction were proved by Glasner and Mozes   \cite{G-M}, Vorobets and Vorobets  \cite{V-V} and Steinberg,Vorobets and Vorobets \cite{SVV}.

We note that the Leary-Minasyan groups $G(A, L')$ defined in \cite{L-M} are HNN extensions with a base group $\mathbb{Z}^n$ and associated subgroups of finite index in $\mathbb{Z}^n$, thus they are generalised Baumslag-Solitar groups of rank $n$. 
 Our main result shows that for  a generalised 
  Baumslag-Solitar group of rank n residual finiteness  is sufficient for self-similarity.
   Thus the residually-finite Leary-Minasyan groups are self-similar.

   \begin{theorem} \label{main1} Let $G$ be a generalized 
  Baumslag-Solitar group of rank n.
    Then the following conditions are equivalent:
  
  i) $G$ is self-similar;
  
  ii) $G$ is residually finite;
  
  iii) $G$ is linear over $\mathbb{Q}$.
  \end{theorem}
 
  A group $G$ is said to be subgroup separable if any two distinct finitely generated subgroups of $G$ have distinct images in some finite quotient of $G$.
  In \cite{M-R-V2} Raptis, Talelli and Varsos showed that the  fundamental group of a graph of groups of polycyclic-by-finite groups with edge groups of finite index in the respective vertex groups is subgroup separable if and only if it is linear 
  over $\mathbb{Z}$. They show that this happens precisely when there is a normal subgroup $H$ of $G$ contained in every edge group as a subgroup of finite index. In the following result we we add more equivalent conditions to the list at the expense of requiring  more restrictions on the graph of groups.

 \begin{theorem}  \label{main2} Let $G$ be the fundamental group of a finite connected graph 
 $(\mathcal{G}, X) $ of finitely
generated free abelian groups with $G_e$ a proper subgroup of $G_{o(e)}$  and $\phi_e(G_e)$   a proper
subgroup of $G_{\tau(e)}$ for every $e \in A$. Furthermore we suppose that all vertex and edge groups have the same rank.   Then the following conditions are equivalent:
  
  i) $G$ is self-similar;
  
  ii) $G$ is residually finite;
  
  iii) $G$ is linear over $\mathbb{Z}$;
  
  iv) $G$ is  subgroup separable.
\end{theorem}

   The above theorem is not true  for fundamental groups of graphs of groups where the embeddings of edge groups into vertex groups have images that are not proper subgroups. For example the classical Baumslag-Solitar group $BS(1,m)$ is metabelian, finitely generated, hence by Hall's  result \cite{Hall}  it is residually finite but for $m \not= \pm 1$  it is not linear over $\mathbb{Z}$. Still this group is self-similar by the results in \cite{Ba-S} or  \cite[Thm. C]{K-Si} or more generally by Theorem \ref{main2} and it is linear over $\mathbb{Q}$.
   
Next we  state a version of Theorem \ref{main2} when $\mathbb{Z}^n$ is substituted with the Heisenberg group $\langle x, y \ | \ [x,y] = z, [z,x] = 1 = [z,y] \rangle$.
   
\begin{theorem} \label{main3} Let $G$ be the fundamental group of a finite connected graph 
 $(\mathcal{G}, X) $ of finitely
generated torsion-free nilpotent groups with $G_e$ a proper subgroup of finite index in  $G_{o(e)}$  and $\phi_e(G_e)$   a proper 
subgroup of finite index in  $G_{\tau(e)}$ for every $e \in A$. We assume further that  one edge group ( and hence every edge group)  $G_e$ is commensurable with the Heisenberg group. Then the following statements are equivalent:

 i) $G$ is self-similar;
  
  ii) $G$ is residually finite;
  
  iii) $G$ is linear over $\mathbb{Z}$;
  
  iv) $G$ is  subgroup separable.

\end{theorem} 
   
   We note that there are finitely generated torsion-free nilpotent groups that are not self-similar \cite{Sa}, \cite{O}. Though a subgroup of a self-similar group is not necessary self-similar, we do not expect that the above theorem holds when Heisenberg group is substituted with an arbitrary finitely generated,  torsion-free nilpotent group.

In Theorem \ref{main1}, Theorem \ref{main2} and Theorem \ref{main3} the self-similarity is established by constructing two virtual endomorphisms or by referring to \cite[Thm. C]{K-Si} where the group is transitive self-similar. Thus in all cases the self-similar groups act on the first level of the corresponding rooted tree by at most two orbits.

    In \cite{M}  Moldovanskii proved that
ascending HNN extensions of a free nilpotent group of finite rank is residually finite.
This was later generalised by Hsu and Wise in \cite{H-W} to ascending   HNN extensions of polycyclic-by-finite groups.
Thus an ascending HNN extension of the Heisenberg group is residually finite and it is interesting to check whether such a group is self similar.

 \begin{theorem} \label{Heisenberg2}  Let $G$ be an  ascending $HNN$-extension 
 of the Heisenberg group. Then $G$ is transitive self-similar. 
\end{theorem}

{\bf Acknowledgements} The author was partially supported by grant CNPq 305457/2021-7 and FAPESP 2024/14914-9.

\section{Preliminaries}
\subsection{Fundamental groups of graph of groups and residual finiteness}

We follow the notations from \cite{R-T-V}.
 Let  $X$ be a finite graph, where loops are allowed
 and $(\mathcal{G}, X) $ be a graph of groups. Let $A$ be an orientation of the edges $E(X)$ of $X$, we write  $o(e)$ and $\tau(e)$ for the beginning and the end of the edge $e \in A$. Thus for each vertex $v$ of the set of vertices $V(X)$ of the graph $X$, there is given a group $G_v$ , for each edge $e \in A$ there is a
subgroup $G_e$ of $G_{o(e)}$ and an isomorphism $$\phi_e: G_e \to \phi_e(G_e) \leq G_{\tau(e)}$$

Let $T$ be a maximal subtree of $X$. Then the fundamental group $G$ of the graph of groups  $(\mathcal{G}, X) $ is denoted by
$\pi(\mathcal{G}, X) $ and is defined in terms of generators and relations by
$$
 G = \pi(\mathcal{G}, X) =\langle \{ G_v \}, \{ t_e \} \ | \ v \in V(X), e \in A, rel(G_v), t_e g_e t_e^{-1} = \phi_e(g_e) \hbox{ for } g_e \in G_e, t_e = 1 \hbox{ for }e \in A \cap E(T) \rangle$$ 

A subgroup $B$ of a group $G$ is isolated if whenever $g^m \in B$ for some $m \in \mathbb{Z} \setminus \{ 0 \}$  then $g \in B$.
If $H$ is a subgroup of  $G$ we denote by  $i_G(H)$  the  isolator of $H$ in $G$, that is, the intersection of all isolated subgroups of $G$ that contain
$H$. 

\begin{theorem} \cite[Thm. 1]{R-T-V} \label{res-fin0}
Let $G$ be the fundamental group of a finite graph 
 $(\mathcal{G}, X) $ of finitely
generated torsion-free nilpotent groups with $G_e$ a proper subgroup of $G_{o(e)}$  and $\phi_e(G_e)$   a proper
subgroup of $G_{\tau(e)}$ for every $e \in A$. Then the following statements are equivalent:

\begin{itemize} 
\item G is residually finite.

\item There is a family of monomorphisms $\{ \rho_v : G_v \to G_v^* \ | \ v \in V(X) \}$ where each $G_v^*$
is a torsion-free nilpotent group and $\rho_v(G_v)$ has finite index in $G_v^*$  such that for each edge $e \in A$
there is an isomorphism
$$\phi_e^* : i_{G_{o(e)}^*} (\rho_{o(e)} (G_e)) \to  i_{G_{\tau(e)}^*} (\rho_{\tau(e)} (\phi_e(G_e)))$$
which extends $\phi_e$,
where $i_{G_{o(e)}^*} (\rho_{o(e)} (G_e))$ ( resp. $i_{G_{\tau(e)}^*} (\rho_{\tau(e)} (\phi(G_e)))$ \ ) is the isolator of $\rho_{o(e)} (G_e)$ ( resp.  $\rho_{\tau(e)} (\phi(G_e))$ \ )  in $G_{o(e)}^*$ ( resp. $G_{\tau(e)}^*$).
\end{itemize}
\end{theorem}

\begin{corollary} \cite[Cor. 1]{R-T-V} \label{res-fin}
Let $G$ be the fundamental group of a finite graph 
 $(\mathcal{G}, X) $ of finitely
generated free abelian groups with $G_e$ a proper subgroup of $G_{o(e)}$  and $\phi_e(G_e)$   a proper
subgroup of $G_{\tau(e)}$ for every $e \in A$. Then the following statements are equivalent:

\begin{itemize} 
\item G is residually finite.

\item There exists a finitely generated free abelian group K, a family of monomorphisms $\{ \rho_v : G_v \to K \ | \ v \in V(X) \}$ and a family of automorphisms $\{\theta_e : K \to K \ | \ e \in A \}$ such that $$\theta_e \rho_{o(e)} |_{G_e} = \rho_{\tau(e)} \phi_e |_{G_e} \hbox{ for all }e \in A.$$
\end{itemize}
 \end{corollary}

 \begin{theorem}\cite{M-R-V2} \label{greek2}
 Let $(\mathcal{G}, X) $ be a finite graph of groups with polycyclic-by-finite vertex groups. Suppose all edge groups are of finite index in the respective vertex groups. Then the following statements are equivalent  for the fundamental group $G$ of the graph of groups   $(\mathcal{G}, X) $ :
 
 1) $G$ is subgroup separable;
 
 2) there is a normal subgroup $H$ of $G$ contained in every edge group as a subgroup of finite index;
 
 3) $G$ is $\mathbb{Z}$-linear.
 \end{theorem}
 
 The classification of the generalized Baumslag-Solitar groups of rank $n$ that are residually finite was
 recently completed by
  Lopez de Gamiz Zearra and  Shepherd \cite{L-S}.
 
 \begin{theorem} \cite{L-S} \label{L-S}
  Let $G$ be a rank $n$ generalised Baumslag-Solitar group. Then $G$ is residually finite if and only if either
$G$ is an ascending HNN extension with a base group $\mathbb{Z}^n$ or
$G$  is virtually $\mathbb{Z}^n$-by-free. \end{theorem}

\subsection{ Self-similar groups and virtual endomorphisms}
Let $\mathcal{T}_m$ be the one rooted $m$-ary tree, i.e. the tree starts with a unique root and every vertex has  $m$ descendents. Let $\mathcal{T}_m^{(0)}, \ldots$,$\mathcal{T}_m^{(m-1)}$  denote the $m$-ary subtrees of $\mathcal{T}_m$ that start at the vertices in the first layer (level) of the tree  $\mathcal{T}_m$. Let $G$ be a group acting on the tree in the way it preserves descendence relation. For every $g \in G$ 
\begin{equation} \label{decomposition} 
g = (g_0, \ldots, g_{m-1}) \sigma
\end{equation}
where $\sigma$ is an element of the  symmetric group $S_m$ that describes the action of $g$ on the first layer (level) of the tree $\mathcal{T}_m$ and each $g_i$ acts on $\mathcal{T}_m$ by fixing its root and all vertices outside the tree  $\mathcal{T}_m^{(i)}$.
By definition the states of the element $g$ are $g_0, \ldots, g_{m-1}$.

 A group $G$ is self-similar if for every $g \in G$ all the states of $g$  belong to $G$ i.e. $G$ is state closed. We say that $G$ is a transitive self-similar group if it acts transitively on the first layer (level) of the tree  $\mathcal{T}_m$. 

A virtual endomorphism is a group homomorphism $$f : H \to G,$$ where $1 <[G : H] < \infty$. Note that $f$ is not required to be injective.  The virtual endomorphism $f$ is called {\bf simple} if whenever $K$ is a normal subgroup  of $G$ such that $K \subseteq H$ and $f(K) \subseteq K$ then $K$ is the trivial group.

\begin{theorem} \cite{N-S}\label{virt-end}  A group
$G$ is a transitive self-similar if and only if there is a simple virtual endomorphism   $f : H \to G$.
\end{theorem}

 Note that by \cite{Brazil2} $\mathbb{Z} \wr \mathbb{Z}$ is not transitive self-similar i.e. it does not admit a simple virtual endomorphism.
There is a generalization of Theorem \ref{virt-end} to intransitive actions given in the following result. 

\begin{theorem} \cite{Brazil} \label{Brazil0} A group
$G$ is a self-similar acting with $k$ orbits on the first level of $\mathcal{T}_m$ if and only if there are virtual endomorphisms   $f_i : H_i \to G$ for $ 1 \leq i \leq k$ such that whenever $K$ is a  normal subgroup of $G$ such that $K \subseteq \cap_{ 1 \leq i \leq k} H_i$ and $f_i(K) \subseteq K$ for $ 1 \leq i \leq k$ then $K$ is the trivial group.
\end{theorem}

Each   one  of the virtual endomorphisms $f_i$ represents one orbit under the action of $G$ on the first level of the tree $\mathcal{T}_m$. 
For more information on self-similar groups we refer the reader to \cite{book}.

Using virtual endomorphisms Kochloukova and Sidki proved the following result

\begin{theorem} \cite[Thm. C]{K-Si} \label{K-S} Let $Q$ be a finitely generated abelian group and $B$ be a finitely generated, right
$\mathbb{Z} Q$-module of Krull dimension 1 such that $C_Q(B) =  \{ q  \in Q  \ | \  B(q-1)= 0 \} = 1$. Then
$G = B \rtimes Q$ is realizable as a transitive self-similar group.
\end{theorem}

 \subsection{Automata generating free products of groups of order 2} \label{free1}

 Let $B_3$ be the Bellaterra automaton
 $$a = (c, b),$$
$$b = (b, c),$$
$$c = (a, a)\sigma.$$ 
 
 Let  $B_4$ be the automaton defined by the wreath recursion
$$a = (c, b),$$ $$ b = (b, c),$$ $$ c = (d, d) \sigma,$$ $$ d = (a, a)\sigma$$ 

For any $ n > 4$ let $B_n$ be the $n$-state automaton defined by the wreath recursion
$$a_n = (c_n, b_n),$$
$$
b_n = (b_n, c_n),$$
$$c_n = (q_{n,1}, q_{n,1})\sigma ,$$
$$q_{n,i} = (q_{n,i+1}, q_{n,i+1})\sigma_{n,i} ,   \ \ i = 1, . . . , n - 5,$$ $$
q_{n,n-4} = (d_n, d_n)\sigma_{n,n-4},$$
$$d_n = (a_n, a_n)\sigma $$
where $\sigma_{n,i}$ are arbitrary elements of the symmetric group $S_2$.
Note that in the recursive description of the second powers of the above generators there is no $\sigma$, hence all these generators have order 2.

\begin{theorem}  \cite{V-V} If $n$ is odd and all $\sigma_{n,i} = \sigma$ the group generated by the automaton $B_n$ for $n \geq 3$ is the free product of $n$ copies of cyclic group of
order 2.
\end{theorem}

\begin{theorem} \cite{S-V} The group generated by the automaton $B_n$ for $n \geq 4$ is the free product of $n$ copies of cyclic group of
order 2.
\end{theorem} 

Note that the group  $F_{n-1}$ generated by $$x_1 = a_nb_n, x_2 = b_n c_n, x_3 = c_n q_{n,1}, \ldots, x_{3+ i} = q_{n,i} q_{n, i+1}, \ldots, x_{n-1}  = q_{n, n-4} d_n$$ is a free group with free basis $x_1, \ldots, x_{n-1}$ as it is a subgroup of $C_2 * \ldots * C_2$. Furthermore
there is a recursive description 
$$x_1 =  (x_2^{-1}, x_2) , x_2 = (x_2 x_3, x_3) \sigma, \ldots $$
that shows that the above automaton generates $F_{n-1}$. Furthermore by appropriately choosing the permutations $\sigma, \sigma_{n,i}$ we get a transitive action  of $F_{n-1}$ on the first level of the binary tree  $\mathcal{T}_2$.
Thus all free finite rank groups are transitive self-similar.

Alternative construction for generating some free groups by automatons can be found in \cite{G-M}.

\section{Proofs}
\begin{proposition} \label{split1} Let $F$ be a free group of finite rank and $A$ a finitely generated  free abelian group. Then any semidirect product $A \rtimes F$ is self-similar with action on the first layer of a respective regular one rooted tree given by two orbits.
\end{proposition}

\begin{proof}
Since $F$ is transitive self-similar there is a simple virtual endomorphism $$f : \widetilde{F} \to F,$$ i.e. $\widetilde{F}$ is a subgroup of finite index in $F$ and if $K$ is a normal subgroup of $F$ such that $K \subseteq \widetilde{F}$ and $f(K) \subseteq K$ then $K = 1$.
Let $$\pi : A \rtimes \widetilde{F} \to \widetilde{F}$$ be the projection with kernel $A$ that is the identity on $\widetilde{F}$ and
$$i : F \to A \rtimes F$$
be the inclusion map. Then consider the composition map
$$\beta = i \circ f \circ \pi : A \rtimes \widetilde{F} \to A \rtimes F$$
Note that $\beta$ is a virtual endomorphism.

Suppose $M$ is a normal subgroup of $A \rtimes F$ that is contained in $A \rtimes \widetilde{F}$ and such that $\beta(M) \subseteq M$. Identifying $F$ with $i(F)$ we have that $\beta(M) \subseteq M$ implies $f(\pi(M)) \subseteq M \cap F  \subseteq (A \rtimes \widetilde{F}) \cap F = \widetilde{F}$ and using $\pi |_{\widetilde{F}} = id_{\widetilde{F}}$ we get $f(\pi(M)) \subseteq M \cap F = \pi(M \cap F) \subseteq \pi(M)$.  Then since $f$ is a simple endomorphism we have that $\pi(M)$ is the trivial group i.e.  $M \subseteq Ker(\pi) = A$. 

Let $$\mu: (2 A) \rtimes F \to A \rtimes F$$
be the virtual endomorphism that is the identity on $F$ and on $2 A$ is multiplication with $1/2$, note we  use additive notation for the group operation in $A$. Suppose that for the normal subgroup $M$ we have  further that $M$ is contained in $(2 A) \rtimes F$ and $\mu(M) \subseteq M$. But $M \subseteq A$ implies $\mu(M) = \frac{1}{2} M \subseteq M$, hence $M$ is the trivial group.
Finally we can apply Theorem \ref{Brazil0} for the virtual endomorphisms $\beta$ and $\mu$ and conclude that $A \rtimes F$ is self-similar.
\end{proof}

\begin{theorem} \label{Heisenberg}  Let $G = N \rtimes F$ be a group with $F$ finitely generated free group and $N = \langle x_1, x_2 \rangle$ the Heisenberg group.  Then $G$ is self-similar  with action on the first layer of a respective regular one rooted tree given by two orbits. Furthermore if $F$ is cyclic then $G$ is transitive self-similar.
\end{theorem}

\begin{proof}   

Assume that $p$ is a fixed prime. Set $$ z = [x_1, x_2] = x_1^{-1} x_2^{-1} x_1 x_2. $$Let $N_1$ be the subgroup of $N$ generated by $ \{ x_i^{p} \ | \ 1 \leq i \leq 2 \}$, note $N_1 \cap N' $ is the subgroup of all element of the commutator $N'$ that are $p^2$-powers and $[N : N_1] = p^4$.  Note that
since $N$ is finitely generated there are finitely many subgroups of $N$ of index $[N : N_1]$ and $F$ permutes these groups ( via its conjugation action on $N$).  Thus there is a subgroup $F_1$ of finite index in $F$ such that $f N_1 f^{-1} = N_1$ for every $f \in F_1$. Let $t_1, \ldots, t_n$ be free generators of $F_1$ as a free group

Consider $G_1 = N_1 \rtimes F_1$. We claim that there is a virtual endomorphism $$f: G_1 \to G$$ given by $f(x_i^{p}) = x_i, f(z^{p^2}) = z$, $f(t_i) = t_i u_i$ for some appropriate $u_i \in N$. 

In order for $f$ to be well defined we need to choose $u_i$ such that \begin{equation} \label{eq1} f(^{ t_i} (x_j^{p})) = ^{f({ t_i})} f( x_j^{p}) \hbox{  for }j = 1,2 \end{equation}
Note that
$$
^{f({ t_i})} f( x_j^{p}) = ^{t_i u_i} f( x_j^{p}) = ^{t_i u_i} x_j = ^{t_i} (x_j (x_j^{-1} u_i x_j u_i^{-1})) = ^{t_i} (x_j)\   ^{t_i} (x_j^{-1} u_i x_j u_i^{-1})$$

Let $$^{t_i} (x_j^p) = x_1^{\alpha_{ij1}} x_2^{\alpha_{ij2}}  z^{\beta_{ij}}$$ for some $\alpha_{ij1}, \alpha_{ij2}, \beta_{ij} \in \mathbb{Z}$. Since $x_j^p \in N_1$ we have $^{t_i} (x_j^p) \in ^{t_i} N_1 = N_1$  and we conclude that
$\alpha_{ij1},\alpha_{ij2} \in p \mathbb{Z}, \beta_{ij} \in p^2 \mathbb{Z}$.

Note that since $^{t_i} (x_j^p) = (^{t_i} x_j)^p$ we have that 
$$ ^{t_i} x_j = x_1^{\frac{\alpha_{ij1}}{p}} x_2^{\frac{\alpha_{ij2}}{p}} b_{i,j}$$
for some $b_{i,j} \in N'$.

On the other hand
$$f(^{ t_i} (x_j^{p})) = f( x_1^{\alpha_{ij1}} x_2^{\alpha_{ij2}} z^{ \beta_{ij}}) =
 f( x_1^{\alpha_{ij1}})  f( x_2^{\alpha_{ij2}})  f( z^{ \beta_{ij}})=  x_1^{\frac{\alpha_{ij1}}{p}}  x_2^{\frac{\alpha_{ij2}}{p}}  z^{\frac{\beta_{ij}}{p^2} } = (^{t_i} x_j) b_{i,j}^{-1}  z^{\frac{\beta_{ij}}{p^2} }
 $$ 
 
 Thus $(\ref{eq1})$ is equivalent to
 $ ^{t_i} (x_j^{-1} u_i x_j u_i^{-1}) =  b_{i,j}^{-1}   z^{\frac{\beta_{ij}}{p^2}} \in N'$ i.e.
$$[x_j, u_i^{-1}] =  x_j^{-1} u_i x_j u_i^{-1} = ^{t_i^{-1}} ( b_{i,j}^{-1}  z^{\frac{\beta_{ij}}{p^2}}) = z^{\gamma_{i,j}} \in N'$$

If $u_i^{-1} = x_1^{a_i} x_2^{b_i}$ then $[x_j, u_i^{-1}] =[x_j, x_1^{a_i} x_2^{b_i}]$, so
$$ z^{\gamma_{i,1}} = [x_1, u_i^{-1}] =[x_1, x_1^{a_i} x_2^{b_i}] = [x_1, x_2^{b_i}] = z^{b_i}$$
$$ z^{\gamma_{i,2}} =[x_2, u_i^{-1}] =[x_2, x_1^{a_i} x_2^{b_i}] = [x_2, x_1^{a_i}] = z^{- a_i}$$
We can choose $a_i = - \gamma_{i,2}$ and $b_i = \gamma_{i,1}$ and so $(\ref{eq1})$ holds.

We claim that for any $g \in \langle x_1^p, x_2^p \rangle$ we have $f(^{ t_i} g) = ^{f({ t_i})} f( g)$. This completes the proof of the fact  that the virtual endomorphism $f$ is well-defined. Indeed if $g = x_1^{m_1 p} x_2^{m_2 p} z^{ m_3 p^2}$, hence
$$ f(^{ t_i} g) = f( (^{ t_i} x_1^p)^{m_1} (^{ t_i} x_2^p)^{m_2} (^{ t_i} z^{p^2})^{m_3}) =
 f( (^{ t_i} x_1^p)^{m_1}) f( (^{ t_i} x_2^p)^{m_2}) f( (^{ t_i} z^{p^2})^{m_3}) $$ $$
=(^{f(t_i)} f(x_1^p))^{m_1} (^{f(t_i)} f(x_2^p))^{m_2} (^{f(t_i)} f(z^{p^2}))^{m_3} = (^{f(t_i)} x_1)^{m_1} (^{f(t_i)} x_2)^{m_2} (^{f(t_i)} z)^{m_3} =
^{f(t_i)} (x_1^{m_1} x_2^{m_2} z^{m_3}) =  ^{f(t_i)} f(g)$$

To complete the proof we construct a map $\beta$ similar to the one of the first paragraph of  the proof of Proposition \ref{split1}. Since $F$ is transitive selfsimilar there is a simple virtual endomorphism $$f_0 : \widetilde{F} \to F.$$
Let $$\pi : N \rtimes \widetilde{F} \to \widetilde{F}$$ be the projection with kernel  $N$ whose restriction  on $\widetilde{F}$ is the identity map and
$$i : F \to N \rtimes F$$
be the inclusion map. Then consider the composition map
$$\beta = i \circ f_0 \circ \pi : N \rtimes \widetilde{F} \to N \rtimes F$$

Suppose $M$ is a normal subgroup of $N \rtimes F$ that is contained in $(N \rtimes \widetilde{F})  \cap (N_1 \rtimes F_1)$ and such that $\beta(M) \subseteq M$ and $f(M) \subseteq M$. Then since  $f_0$ is a simple endomorphism  $M \subseteq Ker(\pi) = N$. By the definition of $f |_N$ and since $f(M) \subseteq M \subseteq N$ it follows that $M = 1$, hence by Theorem \ref{Brazil0} applied for the virtual endomorphisms $f$ and $\beta$ we conclude that $G$ is self-similar  with action on the first layer of a respective regular one rooted tree given by two orbits.

Finally it remains to consider the special case when $F$ is infinite cyclic. Consider the virtual endomorphism $f : G_1 = N_1 \rtimes F_1 \to G$ defined before. If $f$ is simple then $G$ is transitive self-similar as required. Assume that $f$ is not simple. Then there is a normal subgroup $K$ of $G$ such that $K \subseteq G_1$ and $f(K) \subseteq K$. Since $f(N_1) \subseteq N$ we deduce that $f(K \cap N) \subseteq K \cap N \subseteq N_1$, hence $f^{(s)} (K \cap N) \subseteq K \cap N$ for all $s \geq 1$. By the definition of $f |_{N_1}$ we conclude
that $K \cap N = 1$. Hence $K \simeq \mathbb{Z}$, let $t_1$ be a generator of $K$. Since $K$ and $N$ are normal subgroups of $G$ we have $[K, N] \subseteq K \cap N = 1$. Define $G_0 = \langle t_1^p, N_1 \rangle = \langle t_1^p \rangle \times N_1$ a subgroup of finite index in $G$. Define a virtual endomorphism
$$f_0 : G_0 \to G$$
given by $f_0(x_1^p) = x_1, f_0(x_2^p) = x_2, f_0(t_1^p) = t_1$. By the definition of $f_0$ it is obviously simple. Hence $G$ is transitive self-similar.

\end{proof}

The following result  is a version of  Corollary \ref{res-fin}  that does not require that the vertex groups are abelian, only torsion-free nilpotent but we require the extra condition that indices of edge groups in vertex groups are finite. Furthermore the conclusion on $K$ is stronger as $ \widetilde{\rho}_v (G_v)$ has finite index in $K$. Note that Corollary \ref{res-fin}  does not state that ${\rho}_v (G_v)$ has finite index in $K$ as  there indices of edge groups in vertex groups are not necessary finite.

\begin{corollary} \label{res-fin1}
Let $G$ be the fundamental group of a finite connected graph 
 $(\mathcal{G}, X) $ of finitely
generated torsion-free nilpotent groups with $G_e$ a proper subgroup of finite index in  $G_{o(e)}$  and $\phi_e(G_e)$   a proper 
subgroup of finite index in  $G_{\tau(e)}$ for every $e \in A$. Then the following statements are equivalent:

\begin{itemize} 
\item G is residually finite.

\item There exists a finitely generated torsion-free nilpotent group K inside the Malcev´s completion of $G_{v_0}$ for a fixed vertex $v_0 \in V(X)$ (here the Malcev´s completions of $G_v$ for $v \in V(X)$ are  isomorphic), a family of monomorphisms $\{ \widetilde{\rho}_v : G_v \to K \ | \ v \in V(X) \}$ where $ \widetilde{\rho}_v (G_v)$ has {\bf finite } index in $K$ and a family of automorphisms $\{\widetilde{\theta}_e : K \to K \ | \ e \in A \}$ such that $$\widetilde{\theta}_e \widetilde{\rho}_{o(e)} |_{G_e} = \widetilde{\rho}_{\tau(e)} \phi_e |_{G_e} \hbox{ for all }e \in A.$$
\end{itemize}
 \end{corollary}

\begin{proof} We apply Theorem \ref{res-fin0} i.e. residual finiteness is equivalent to the second condition of Theorem \ref{res-fin0}. Our second condition is a stronger version of the second condition of Theorem \ref{res-fin0}.

 We suppose that the second condition of Theorem \ref{res-fin0} holds i.e. there is a family of monomorphisms $$\{ \rho_v : G_v \to G_v^* \ | \ v \in V(X) \}$$ where $G_v^*$
is torsion-free nilpotent and $\rho_v(G_v)$ has finite index in $G_v^*$  such that for each edge $e \in A$
there is an isomorphism
$$\phi_e^* : i_{G_{o(e)}^*} (\rho_{o(e)} (G_e)) \to  i_{G_{\tau(e)}^*} (\rho_{\tau(e)} (\phi_e(G_e)))$$
which extends $\phi_e$,
where $i_{G_{o(e)}^*} (\rho_{o(e)} (G_e))$ ( resp. $i_{G_{\tau(e)}^*} (\rho_{\tau(e)} (\phi(G_e)))$ \ ) is the isolator of $\rho_{o(e)} (G_e)$ ( resp.  $\rho_{\tau(e)} (\phi(G_e))$ \ )  in $G_{o(e)}^*$ ( resp. $G_{\tau(e)}^*$).

Since all edge groups have finite indices in the corresponding vertex groups and by considering $G_v^*$ minimal possible via inclusion  we deduce that all  isolators considered in $G_v^*$ are equal to $G_v^*$. Thus there is an isomorphism
$$\phi_e^* : G_{o(e)}^* \to G_{\tau(e)}^*$$
that extends $\phi_e$ for every $ e \in A$. By the proof of Theorem \ref{res-fin0} the group $G_v^*$ is inside the Malcev´s completion of $G_v$ for all $v \in V(X)$.

Since $X$ is a connected graph each vertex $v \in V(X)$ is beginning or end of some edge $e \in A$ ( or both), in particular all groups $G_v^*$, for $v \in V(X)$, are isomorphic. Let $K$ be a fixed group inside the Malcev´s completion of $G_{v_0}$  together with isomorphisms $\alpha_v : G_v^* \to K$ for all $v \in V(X)$. Then we can set $$\widetilde{\rho}_v = \alpha_v \rho_v: G_v \to K$$   and $$\widetilde{\theta}_e = \alpha_w \phi_e^* \alpha_v^{-1}  : K \to K \hbox{ for } v = o(e), w = \tau(e).$$

\end{proof}

{\bf Proof of Theorem \ref{main3}}  $i) \implies ii)$ By  the definition of self-similar groups as acting faithfully on a rooted tree they are always residually finite.

 $iii) \implies ii)$  All finitely generated linear groups over $\mathbb{Z}$ are residually finite. In fact by a result of Malcev a finitely generated linear group ( over a commutative ring) is residually finite \cite{Mal}  but the converse does not hold, by a result of Drutu-Sapir \cite{D-S}  the group $\langle a, t | a^{t^2}=a^2 \rangle$ is residually finite and non-linear.

$ii) \implies i)$
Assume now that $G = \pi (\mathcal{G}, X)$ is residually finite. Then by Corollary \ref{res-fin1} there is a finitely generated torsion-free nilpotent group $K$ (recall $K$ is inside the Malcev´s completion of $G_{v_0}$ for a fixed vertex $v_0 \in V(X)$), a family of monomorphisms $$\{ \widetilde{\rho}_v : G_v \to K \ | \ v \in V(X) \}$$ where $ \widetilde{\rho}_v (G_v)$ has {\bf finite } index in $K$ and a family of automorphisms $\{\widetilde{\theta}_e : K \to K \ | \ e \in A \}$ such that $$\widetilde{\theta}_e \widetilde{\rho}_{o(e)} |_{G_e} = \widetilde{\rho}_{\tau(e)} \phi_e |_{G_e} \hbox{ for all }e \in A.$$
 Note that since $K$ is a finitely generated subgroup of the Marcev´s completion of $G_{v_0}$ and $G_{v_0}$ is commensurable to the Heisenberg group we conclude that $K$ is commensurable to the Heisenberg group, hence  is nilpotent of class 2, torsion-free of Hirsh length 3. We split the proof in several steps.

1)
Let $F$ be the subgroup of $G$ generated by the set $$\{ t_e  \ | \  e \in A \cap (E(X) \setminus E(T)) \}.$$ Then by the definition of the fundamental group of a graph of groups $F$ is a free group with a free basis $\{ t_e  \ | \  e \in A \cap (E(X) \setminus E(T)) \}$.
Let
$$K \rtimes F$$
be the semi-direct product where  the conjugation action of $t_e$  on $K$ is given by $\widetilde{\theta}_e$. Then there is a group homomorphism
$$\gamma: G \to K \rtimes F$$
whose restriction on each $G_v$ is $\widetilde{\rho}_v$  and  $\gamma(t_e) = t_e$ for each $e \in A \cap (E(X) \setminus E(T))$.

Set $$F_0 = Ker(\gamma).$$ By construction $F_0$ intersects trivially each vertex group $G_v$, hence by standard Bass-Serre theory $F_0$ is a free group. 

2) We can suppose that for each $v \in V(X)$ the monomorphism $\widetilde{\rho}_{v}$ is an inclusion map i.e. $G_v$ is a subgroup of finite index in $K$. Furthermore since in $G$ for $e \in A \cap E(T)$ we have that $G_e = G_{o(e)} \cap G_{\tau(e)}$ we can assume that $\widetilde{\theta}_e$ and $\phi_e$ are the identity maps in this case.

Consider the subgroup of $G$ defined by 
$$D = \cap_{e \in A} G_e$$ Note that $D$ has finite index in $G_e$, hence in $K$. Let $s > 0$ be such a natural number that 
$$\langle g^s \ | \ g \in K \rangle = K^s \subseteq D$$ Since $\widetilde{\theta}_e$ is an automorphism of $K$, it is an automorphism of $K^s$.  
We set $H = K^s$ a subgroup of $G$.

Let $ e \in A \cap (E(X) \setminus E(T))$. Note that $H = K^s \leq G_e \leq K$. Since $\widetilde{\rho}_{o(e)}$ is the inclusion map we have that the restriction of $\widetilde{\theta}_e \in Aut(K)$ on $G_e$ is the map $\phi_e$ ( and this is conjugation with $t_e$ in $G$). Thus in the group $G$ we have 
$$t_e H t_e^{-1} = t_e (K^s) t_e^{-1}  = \phi_e( K^s) = \widetilde{\theta}_e (K^s) =  \widetilde{\theta}_e(K)^s = K^s = H$$
Thus we conclude that $H$ is invariant under conjugation with $t_e$. By construction $H$ is a subgroup of each $G_e$, hence of every vertex group $G_{o(e)}$.
Note that $H =K^s$ is a characteristic subgroup of $K$, hence normal in $K$ and hence in each $G_{o(e)}$. Since the graph is connected  we deduce that
 $H$ is a normal subgroup of $G$.

Note that this argument shows that $H$ is normal in $K \rtimes F$ too. Thus $\gamma$ induces a homomorphism $G/ H \to  (K/ H) \rtimes F$, where $K/ H$ is finite.

Since $F_0 \cap H  \subseteq F_0 \cap G_e = 1$ and both $F_0$ and $H$ are normal in $G$ they commute i.e. $[F_0, H] = 1$, thus $H F_0 = H \times F_0$

3) Define the following subgroup of $G$
\begin{equation} \label{recall} \widetilde{G} = (H \times F_0) \rtimes F = H \rtimes \widetilde{F} \end{equation}
where $\widetilde{F} = F_0 \rtimes F$.
Since $K/ H$ is finite, $\widetilde{G}$ is a subgroup of {\bf finite} index in $G$.

We claim that $$\widetilde{F}  \hbox{ is free.}$$ It suffices to show that $\widetilde{F} \cap G_v = 1$ for every vertex $v$.
Suppose $g \in \widetilde{F} \cap G_v$, $g = x_0 x$, where $x_0 \in F_0,x \in F$. Since $g \in G_v$ we have $\gamma (g) \in K$ but $\gamma(g) = \gamma(x_0) \gamma(x) = \gamma(x) = x \in F$. Since $K \cap F = 1$ we get $x \in F \cap K = 1$, so $g = x_0 x = x_0 \in G_v \cap F_0 = 1$.

4)  We claim that $\widetilde{G}$ is self-similar. Indeed $\widetilde{G}$ has finite index in $G$, hence is finitely generated and its quotient $\widetilde{F}$ is finitely generated free group. By construction $H$ is torsion-free and  commensurable with the Heisenberg group. But since every subgroup of finite index in a Heisenberg group contains a subgroup of finite index that is isomorphic to the Heisenberg group we deduce that there is a subgroup of finite index $H_1$ of $H$ such that $H_1$ is isomorphic to the Heisenberg group. Since there are only finitely many subgroups of $H$ of index $[H: H_1]$ there is a subgroup of finite index $F_1$ of $\widetilde{F}$ that normalizes $H_1$ and $$G_1  = H_1 \rtimes F_1 \hbox{ is a subgroup of finite index in }G$$
By Theorem \ref{Heisenberg} $G_1$ is self-similar, hence $G$ is self-similar. Indeed any finite set of virtual endomorphisms of $G_1$ that shows that $G_1$ is self-similar ( i.e. satisfy Theorem \ref{Brazil0} for the group $G_1$) can be viewed as virtual endomorphisms of $G$, hence $G$ is self-similar.

$ii) \implies iii), iv)$  It remains to show that if $G$ is residually finite then it is linear over $\mathbb{Z}$ and subgroup separable.  Recall equation ( \ref{recall}) and $H$ commensurable with a Heisenberg group, $F$ acts on $H$ via $\{ \widetilde{\theta}_e \}$. Note that $H$ is contained in all edge groups and is normal in $G$. Then by  Theorem \ref{greek2} $G$ is linear over $\mathbb{Z}$ and subgroup separable. 

$ iv) \implies ii)$ holds for an arbitrary group.

 \medskip

  {\bf Proof of Theorem \ref{main2}}
  The proof is  the same as the proof of Theorem \ref{main3}  but we use Proposition \ref{split1} instead of Theorem \ref{Heisenberg}.
   
   \medskip

  {\bf Proof of Theorem \ref{main1} }
  $i) \implies ii)$ Every self-similar group is residually finite.
  
  $ii) \implies i)$ By Theorem \ref{L-S} $G$ is either virtually $\mathbb{Z}^n$-by-free or an ascending HNN extension of $\mathbb{Z}^n$. In the former case we can apply Proposition \ref{split1} to deduce that $G$ is self-similar. 
  
In the latter case we observe that $G = A \rtimes Q$, $Q \simeq \mathbb{Z}$ and $A$ is  a strictly  increasing union of copies of $\mathbb{Z}^n$. We claim that $A$ has Krull dimension 1 as a $\mathbb{Z} Q$-module where $Q$ acts on $A$ by conjugation,  hence by Theorem \ref{K-S} $G$ is self-similar since $Q$ does not centralizes $A$. 

To prove that  $Krull dim (A) = 1$ note that $$Krull dim (A) = Krull dim (\mathbb{Z} Q/ I)$$ where $I$ is the annihilator of $A$ in $\mathbb{Z} Q$. Since $0$ is a prime ideal of $\mathbb{Z} Q$ and $I \not= 0$ we have that $$Krull dim (\mathbb{Z} Q/ I) \leq Krull dim (\mathbb{Z} Q) - 1 = 1$$ If $Krull dim (\mathbb{Z} Q/ I) = 0$ then all prime ideals $P$ of $\mathbb{Z} Q$ such that $I \subseteq P$ should be maximal. But any maximal ideal of $\mathbb{Z} Q$ has finite index in $\mathbb{Z} Q$ ( as an additive subgroup). Such a fact is true even for more general groups as polycyclic groups ( see  \cite{Ros}). Since $\mathbb{Z} Q$ is a Noetherian ring by standard commutative algebra argument there are only finitely many prime ideals $P_1, \ldots, P_m$ minimal above $I$ and $$I \subseteq \sqrt{I} = P_1 \cap \ldots \cap P_m$$ In particular $\sqrt{I}$ has finite index in $\mathbb{Z} Q$, hence $I$ has finite index in $\mathbb{Z} Q$. And since $A$ is finitely generated as a $\mathbb{Z} Q$-module we conclude that $A$ is finite, a contradiction.
  
 $ii) \implies iii)$
  
1) If  $G$ is  virtually $\mathbb{Z}^n$-by-free we claim that $G$ is linear over $\mathbb{Z}$. Indeed let $\widetilde{G} = H \rtimes  \widetilde{ F}$ be a subgroup of finite index in $G$, where $ \widetilde{ F}$ is a finite rank free group and $H \simeq \mathbb{Z}^n$. Then $ \widetilde{ F}$ acts on $H$ by conjugation. 

Let $F_0$ be the kernel of this action, hence $F_0$ is a normal subgroup of $\widetilde{G}$. Then $\widetilde{G}/ F_0 \leq \mathbb{Z}^n \rtimes GL_n(\mathbb{Z})$.  The group $\mathbb{Z}^n \rtimes GL_n(\mathbb{Z})$ embeds in $GL_{n+1}(\mathbb{Z})$ where the matrix $A = ( a_{ij}) \in GL_n(\mathbb{Z})$ is represented by a matrix $B = (b_{ij})$  defined by $b_{ij} = a_{ij}$ for $i,j \leq n$, $b_{n+1,i} = 0 = b_{i, n+1}$ for $ 1 \leq i \leq n$ and $b_{n+1, n+1} = 1$ and $h = (z_1, \ldots, z_n)$ is represented by a matrix with only 1 on the main diagonal, the last column is $(z_1, \ldots, z_n, 1)^t$ and the other entries are 0. 
  
  On the other hand $\widetilde{G} / H \simeq \widetilde{F}$ is free, hence it is linear over $\mathbb{Z}$ i.e. embeds in $GL_m(\mathbb{Z})$ for appropriate $m$ and $\widetilde{G} / F_0 $ embeds in $ H \rtimes GL_n(\mathbb{Z})$, which by the above embeds in $GL_{n+1}(\mathbb{Z})$. Since $H \cap 
  F_0  = 1$ we deduce that $\widetilde{G} = H \rtimes \widetilde{F}$ embeds in $(\widetilde{G}/ H) \times (\widetilde{G} / F_0) $ that itself embeds in $GL_m(\mathbb{Z}) \times GL_{n+1}(\mathbb{Z}) \leq GL_{m + n+1}(\mathbb{Z})$. Then $V = \mathbb{Z}^{m+n+1}$ is $\widetilde{G}$-module and the induced $G$-module $\mathbb{Z} G \otimes_{\mathbb{Z} \widetilde{G}} V  = \mathbb{Z}^k$ for $k = [G : \widetilde{G}] (m+ n+1)$ gives an embedding of $G$ in $GL_k(\mathbb{Z})$.
  
2) If $G$ is an ascending HNN extension of $\mathbb{Z}^n$ we claim that $G$  is linear over $\mathbb{Q}$. Indeed  $G = \langle A \simeq \mathbb{Z}^n, t \ | \ t a t^{-1} = \varphi (a), a \in A \rangle
$ where $\varphi : A \to A$ is a monomorphism. If $\varphi$ is surjective then by the previou paragraph $G$ is linear over $\mathbb{Z}$. If $\varphi$ is not surjective we have that $B = \cup_{i \geq 0} t^{-i} A t^i$ is a nested increasing union  where each $t^{-i} A t^i \simeq A \simeq \mathbb{Z}^n$, hence $B$ is a normal subgroup of $G$ and $B$  is embeddable in $\mathbb{Q}^n$, in particular $G$ is embeddable in $\mathbb{Q}^n \rtimes \mathbb{Z}$, that is embeddable in $\mathbb{Q}^n \rtimes GL_n(\mathbb{Q})$. By the argument of the previous paragraph with $\mathbb{Z}$ substituted by $\mathbb{Q}$ we have that $\mathbb{Q}^n \rtimes GL_n(\mathbb{Q})$ is embeddable in $ GL_{n+1}(\mathbb{Q})$.

$iii) \implies ii)$ By Malcev \cite{Mal}  every finitely generated group that is  linear  over a commutative ring is residually finite.

\section{Ascending HNN extension of the Heisenberg group is self-similar}

{\bf Proof of Theorem \ref{Heisenberg2}}

 $G = \langle N, t \ | \ ^t g = \varphi (g) \hbox{ for } g \in N \rangle$ where  $N = \langle x, y \rangle$ is the Heisenberg group and  $\varphi : N \to N$ is a monomorphism. Since the case when $\varphi$ is an automorphism  was resolved in Theorem \ref{Heisenberg} we can assume that $\varphi$ is not surjective. Note that $\varphi(N)$ has finite index in $N$.
 
  Let $$\varphi(x) \in x^{a_1} y ^{b_1} Z \hbox{ and } \varphi(y) \in x^{a_2} y ^{b_2} Z,$$ where $Z = \langle z \rangle$  is the centre of $N$, $[x,y] = z$ and   $A = \begin{bmatrix}
 a_1 & a_2 \\
 b_1 & b_2 \\
 \end{bmatrix}$ is $2 \times 2$ matrix with integer coefficients.
 
  Since $\varphi(N) \not= N$ we have that $det(A) \not= \pm 1$.
 
 \medskip
 {\bf Claim 1}      There exists a subgroup of finite index $\langle x_1, y_1 \rangle$ in $N$, an integer $k \geq 1$  and a positive  odd prime $p$ where     $2 det(A)$ divides $p-1$  such that for  $N_0 = \langle x_1, y_1 \rangle$, $N_1 = \langle x_1^p, y_1^p \rangle$ we have that $\varphi^k (N_0) \subseteq N_0, \varphi^k (N_1) \subseteq N_1$.

 \medskip
{\bf Proof of the claim.}  
1)  Suppose  $det(A) - tr(A) + 1 \not= 0$ and consider  $x_1 = x z ^{\alpha}, y_1 = y z^{\beta}$ for  $\alpha, \beta \in \mathbb{Z}$. We will specify $\alpha$ and $\beta$ later.

Since $\varphi(x^p) = \varphi(x)^p $, $\varphi(y^p) = \varphi(y)^p$ we can write $$\varphi(x^p) = x^{pa_1} y^{pb_1} z^{pc_1}, \varphi(y^p) = x^{pa_2} y^{pb_2} z^{pc_2}$$ for some $c_i \in \mathbb{Z}$.
 Then using that $\varphi(z) = \varphi([x,y]) = z^{det(A)}$ we have
 \begin{equation} \label{q1} \varphi((x z ^{\alpha})^p) = \varphi( x^p z^{\alpha p}) = \varphi(x^p) \varphi(z)^{\alpha p} = x^{pa_1} y^{pb_1} z^{pc_1} z^{ \alpha p . det(A)} \end{equation}
 and 
  \begin{equation} \label{q2} \varphi((y z ^{\beta})^p) = \varphi( y^p z^{\beta p}) = \varphi(y^p) \varphi(z)^{\beta p} = x^{pa_2} y^{pb_2} z^{pc_2} z^{ \beta p . det(A)} \end{equation}

  Note that (\ref{q1}), (\ref{q2}) and the fact that $ [x_1^p,y_1^p]  = z^{p^2}$ implies that $\varphi((x z ^{\alpha})^p), \varphi((y z ^{\beta})^p) \in N_1$ if and only if
 $$
 \varphi((x z ^{\alpha})^p) \in (x z^{\alpha})^{p a_1} (y z^{\beta})^{p b_1} Z^{p^2} = 
 x^{p a_1} y^{p b_1} z^{\alpha p a_1 + \beta p b_1} Z^{p^2} 
 $$ 
 and 
 $$
 \varphi((y z ^{\beta})^p) \in (x z^{\alpha})^{p a_2} (y z^{\beta})^{p b_2} Z^{p^2} = 
 x^{p a_2} y^{p b_2} z^{\alpha p a_2 + \beta p b_2} Z^{p^2} 
 $$
 This is equivalent to $$\alpha  a_1 + \beta  b_1 - \alpha. det(A) = c_1 \ ( \ mod \ p ), \alpha  a_2 + \beta  b_2 - \beta . det(A) = c_2 \ ( \ mod \  p )$$
 This is a congruence system with variables $\alpha, \beta$ and determinant  $$(a_1 - det(A)) (b_2 - det(A)) - a_2 b_1 = a_1 b_2 - a_2 b_1 - (a_1 + b_2) det(A) + det(A)^2 =$$ $$ det(A)^2 - tr(A) det(A)  + det(A) = det(A) ( det(A) - tr(A) + 1)$$
 
  We can choose $p$ such that $$det(A) ( det(A) - tr(A) + 1) \not\eq 0 ( \ mod \ p)$$ and can always resolve the congruence system. In this case we define $k = 1$.

2)  Finally if $det(A) - tr(A) + 1 = 0$  then for $\lambda_1, \lambda_2$ the e-values of $A$ we have $(\lambda_1 - 1) ( \lambda_2 - 1) =\lambda_1 \lambda_2 - ( \lambda_1 + \lambda_2) + 1 = 0$. Then $\lambda_1 = 1$ or $\lambda_2 = 1$, say $\lambda_1 = 1$.
 Then there is a subgroup of finite index $N_0 = \langle x_0, y_0 \rangle \leq N$
such that
$$\varphi(x_0) = x_0 z^{d_1},  \  \varphi(y_0) = y_0^{\lambda} z^{d_2}$$ where $\lambda = \lambda_2, d_1, d_2 \in \mathbb{Z}$. Let $z_0 = [x_0, y_0] = z^m$ for some $m \in \mathbb{Z}$. Note that $\varphi(x_0)$ and $\varphi(y_0)$ are not necessary elements of $N_0$.

 Let $p$ be a prime  and set $k = p |m|$.

 Note that  since $\varphi(z) = \varphi([x,y]) = z^{det(A)} = z^{\lambda}$ for every $s \geq 1$ we have $\varphi^s(y_0) = y_0^{\lambda^s} z^{s d_2 \lambda^{s-1}}$,in particular $$\varphi^k(y_0) = y_0^{\lambda^k} z^{p| m| d_2 \lambda^{k-1}} =  y_0^{\lambda^k} z_0^{ \pm p  d_2 \lambda^{k-1}}  \in N_0$$ Similarly   $$\varphi^k(x_0)= x_0 z^{p |m| d_1 } = x_0 z_0^{ \pm p d_1} \in N_0$$

 Then if we substitute $N$ with $N_0$ i.e. $x$ with $x_0$, $y$ with $y_0$,$z$ with $z_0$, we substitute $A$ with $A^k$ and $\varphi$ by $\varphi^k$   we can solve the congruence system from case 1)  for $a_1 = 1, b_1 = 0, c_1 = \pm p d_1 ,  a_2 = 0, b_2 = \lambda^k $, $c_2 = \pm p d_2 \lambda^{k-1} = 0 \ ( \ mod \ p \ )$
 as the congruence system in this case is
 $$\alpha  ( 1 - det(A^k)) \equiv c_1 \equiv 0 \ ( \ mod \ p ), \beta.0 \equiv c_2 \equiv 0 \ ( \ mod \  p )$$
 Thus we can take $\alpha = \beta = 0$ and $x_1 = x_0, y_1 = y_0$. 

  In both cases 1) and 2)  we can resolve the congruence system but in case 1) we have to avoid  finitely many primes $p$. Finally we need one odd prime $p$ such that $2 det(A)$ divides $p-1$ but by the Dirichlet Theorem the arithmetic progression $\{ 2 det(A) s+1 \ | \ s \in \mathbb{Z}, det(A) s > 0 \}$, always contains infinitely many prime numbers $p$. Hence there is a prime $p$ that satisfies Claim 1.
 This completes the proof of Claim 1.
 
 \medskip
 {\bf Claim 2} {\it Suppose $K$ is a subgroup of finite index in $N$ such that $\varphi^k(K) \subseteq K$ for some $k \geq 1$
 . Then the subgroup $G_K$ of $G = \langle N, t \rangle$ generated by $K$ and $t^k$ has finite index in $G$. In particular if $G_K$ is transitive self-similar then $G$ is transitive self-similar.}
 
 \medskip
 {\bf Proof}  Let $M$ be the normal closure of $N$ in $G$. Then $M$ is inside the Malcev completion of $N$, $M$ is a torsion-free nilpotent group of nilpotent class 2 and since $\varphi(z) = z^{det(A)}$ we have in additive notation that  $Z(M) \simeq \mathbb{Z}[\frac{1}{det(A)}]$ where $z$ corresponds to $1 \in \mathbb{Z}$.  Note that $M/ Z(M) \simeq \cup_{i \geq 1} A^{-i} \mathbb{Z}^2 = \cup_{i \geq 1} A^{-ik} \mathbb{Z}^2$.
 
 We set $M_1 = G_K \cap  M$. To complete the proof it suffices to show that $M_1$ has finite index in $M$.
 
 Since $K$ has finite index in $N$  we have that  $(m \mathbb{Z})^2 \leq K/ (Z(N) \cap K) \leq  N/ Z(N) =\mathbb{Z}^2 $ for some integer $m \geq 1$. Then the image of $M_1$ in $M/ Z(M)$ contains
 $ \cup_{i \geq 1} A^{-ik} (m\mathbb{Z})^2$. Since $A^{-ik} (m\mathbb{Z})^2$ has index $m^2$ in 
 $ A^{-ik} \mathbb{Z}^2$ we conclude that  $ \cup_{i \geq 1} A^{-ik} (m\mathbb{Z})^2$ has index
 at most $m^2$ in $\cup_{i \geq 1} A^{-ik} \mathbb{Z}^2 \simeq M/ Z(M)$. Thus the image of $M_1$ in $M/ Z(M)$ has finite index in  $M/ Z(M)$.
 
  Since $K$ has finite index in $N$  we have that $Z(N) \cap K =\langle z^s \rangle$ for some integer $s \geq 1$. Then $Z(M) \cap M_1$ contains  $\cup_{i \geq 1} \langle ^{t^{ik}} z^s \rangle$. Moving to additive notation $s \mathbb{Z}[\frac{1}{det(A)^k}] \leq Z(M) \cap M_1 \leq Z(M) = \mathbb{Z}[\frac{1}{det(A)}] = \mathbb{Z}[\frac{1}{det(A)^k}]$. But for any integer
$r \not= 0$ we have that $[\mathbb{Z}[\frac{1}{r}] :  s \mathbb{Z}[\frac{1}{r}]] \leq s$. Hence 
 $Z(M) \cap M_1$ has finite index in $Z(M)$.
 
Finally if $G_K$ is transitive self-similar then there is corresponding simple virtual endomorphism of $G_K$, that can be considered as a simple virtual endomorphism of $G$ and we deduce that $G$ is transitive self-similar.
This completes the proof of Claim 2.

 \medskip

 {\bf Claim 3} $G$ is  a transitive self-similar group.
 
 {\bf Proof of Claim 3}
 Consider  the group $N_0 = \langle x_1, y_1 \rangle$ and the integer $k \geq 1$ given by Claim 1.  By Claim 2 applied for $K = N_0$ it suffices to consider the case when $k = 1$ and $N_0 = N$ i.e. $x_1, y_1$ are generators of the Heisenberg group $N$ with $[x_1, y_1] = z$ a generator of $Z(N)$ and $\varphi(N) \subseteq N$ and $\varphi(\langle x_1^p, y_1^p \rangle) \subseteq \langle x_1^p, y_1^p \rangle = N_1$.
 
 Let $G_1$ be the subgroup of $G$ generated by $x_1^p, y_1^p$ and $t$. By Claim 2 $G_1$ has finite index in $G$.

 We claim that there is a simple  virtual endomorphism $$f : G_1 \to G$$ given by $$f(x_1^p) = x_1, f(y_1^p) = y_1, f(t) = tu$$
  for some $u \in M$. We have to check $f$ is well-defined i.e. \begin{equation} \label{q10}
 ^{f(t)} f(x_1^p) = f(^t x_1^p) \hbox{ and }  ^{f(t)} f(y_1^p) = f(^t y_1^p) \end{equation}
 Indeed $ \ ^t ( x_1^p) \in ^t N_1 = \varphi(N_1)$, hence $$^t (x_1^p) = x_1^{p \gamma_{11}} y_1^{p \gamma_{12}} z^{p^2 \gamma_{13}}, \hbox{ similarly } ^t (y_1^p) = x_1^{p \gamma_{21}} y_1^{p \gamma_{22}} z^{p^2 \gamma_{23}} \hbox{ for some } \gamma_{ij} \in \mathbb{Z}$$
 Hence $$^t x_1 = x_1^{\gamma_{11}} y_1^{\gamma_{12}} z_1 \hbox{ and } ^t y_1 = x_1^{\gamma_{21}} y_1^{\gamma_{22}} z_2, \hbox{ for some } z_1, z_2 \in \langle z \rangle.$$

  Note that $[x_1, y_1] = z$, so
 $$x_1^{p \gamma_{11}} y_1^{p \gamma_{12}} z^{p^2 \gamma_{13}} = ^t ( x_1^p) = (^t x_1)^p =
 (x_1^{\gamma_{11}} y_1^{\gamma_{12}} z_1)^p = x_1^{p \gamma_{11}} y_1^{p \gamma_{12}} z^{- \gamma_{11} \gamma_{12} \frac{p(p-1)}{2}} z_1^p$$
 This implies that \begin{equation} \label{eqeq7} z^{p \gamma_{13}+ \gamma_{11}\gamma_{12} \frac{(p-1)}{2}} = z_1 \end{equation} 
 Similarly
 $$ x_1^{p \gamma_{21}} y_1^{p \gamma_{22}} z^{p^2 \gamma_{23}} = ^t (y_1^p) = (^t y_1)^p =
 (x_1^{\gamma_{21}} y_1^{\gamma_{22}} z_2)^p = 
 x_1^{\gamma_{21} p} y_1^{\gamma_{22} p} z^{- \gamma_{21} \gamma_{22} \frac{p (p-1)}{2}} z_2^p
 $$
 This implies that \begin{equation} \label{eqeq8} z^{p \gamma_{23}+ \gamma_{21}\gamma_{22} \frac{(p-1)}{2}} = z_2 \end{equation}  
 
 Then $$f(^t x_1^p) = f( x_1^{p \gamma_{11}} y_1^{p \gamma_{12}} z^{p^2 \gamma_{13}}) = 
  f( x_1^{p \gamma_{11}}) f( y_1^{p \gamma_{12}}) f( z^{p^2 \gamma_{13}}) =   x_1^{ \gamma_{11}}  y_1^{ \gamma_{12}}  z^{ \gamma_{13}} $$
  and
  $$f(^t y_1^p) = f( x_1^{p \gamma_{21}} y_1^{p \gamma_{22}} z^{p^2 \gamma_{23}}) = 
  f( x_1^{p \gamma_{21}}) f( y_1^{p \gamma_{22}}) f( z^{p^2 \gamma_{23}}) =   x_1^{ \gamma_{21}}   y_1^{ \gamma_{22}} z^{ \gamma_{23}}$$
  On the other hand
  $$^{f(t)} f(x_1^p)  = ^{tu} x_1    = ^t x_1 (^{t} (x_1^{-1} u x_1 u^{-1}))   = x_1^{\gamma_{11}} y_1^{\gamma_{12}} z_1 (^{t} (x_1^{-1} u x_1 u^{-1}))  $$
  and
   $$^{f(t)} f(y_1^p)  = ^{tu} y_1    = ^t y_1 (^{t} (y_1^{-1} u y_1 u^{-1}))   = x_1^{\gamma_{21}} y_1^{\gamma_{22}} z_2 (  ^{t} (  y_1^{-1} u y_1 u^{-1}))  $$
   Hence (\ref{q10}) is equivalent to
   \begin{equation} \label{eqeq9} x_1^{-1} u x_1 u^{-1} = ^{t^{-1}} (z_1^{-1} z^{\gamma_{13}})  =  (z_1^{-1} z^{\gamma_{13} })^{1/ det(A)} \end{equation}  and \begin{equation} \label{eqeq10}   y_1^{-1} u y_1 u^{-1} = ^{t^{-1}} (z_2^{-1} z^{\gamma_{23} })  =  (z_2^{-1} z^{\gamma_{23} }) ^{1/ det(A)} .\end{equation}
 By  (\ref{eqeq7})   and since $det(A)$ divides $(p-1)/2$ we get
\begin{equation}  \label{eqeq11}
(z_1^{-1} z^{\gamma_{13} })^{\frac{1}{det(A)}} = (z^{(1-p) \gamma_{13}- \gamma_{11}\gamma_{12} \frac{(p-1)}{2} })^{\frac{1}{det(A)}} \in \langle z \rangle
\end{equation}
Similarly by (\ref{eqeq8})
\begin{equation} \label{eqeq12}
(z_2^{-1} z^{\gamma_{23} })^{\frac{1}{det(A)}} = (z^{(1-p) \gamma_{23}- \gamma_{21}\gamma_{22} \frac{(p-1)}{2} })^{\frac{1}{det(A)}} \in \langle z \rangle
\end{equation}

  Take $u = x_1^{\alpha_0} y_1^{\beta_0}$,so
    \begin{equation} \label{eqeq13} x_1^{-1} u x_1 u^{-1} = [x_1, u^{-1}] = [x_1, y_1^{- \beta_0}] = z^{- \beta_0}\end{equation} and \begin{equation} \label{eqeq14} y_1^{-1} u y_1 u^{-1} =  [y_1, u^{-1}] = [y_1, x_1^{- \alpha_0}] = z^{ \alpha_0} \end{equation}
   
 Then by (\ref{eqeq9}),(\ref{eqeq10}),(\ref{eqeq11}),(\ref{eqeq12}),(\ref{eqeq13}) and (\ref{eqeq14})   $f$ will be well-defined if we take    
 $$
 \alpha_0 = \frac{1}{det(A)}( (1-p)\gamma_{23} - \gamma_{21}\gamma_{22} \frac{(p-1)}{2} ) \in \mathbb{Z}
 $$
 and 
 $$\beta_0 = \frac{1}{det(A)}(- (1-p)\gamma_{13} + \gamma_{11}\gamma_{12} \frac{(p-1)}{2} )
   \in \mathbb{Z}$$

 Let $M_1$ be the normal closure of $N_1$ in $G_1$.  Note that $f |_{M_1} : M_1 \to M$ is injective, hence $f$ is injective. 
 
 We claim that $f$ is a simple virtual endomorphism. If not  there 
 is a non-trivial normal subgroup $K$ of $G$ such that $K \subseteq G_1$ and  $f(K) \subseteq K$. Then $f(K \cap M) \subseteq K \cap M$.
 
 Suppose $g \in K \cap Z(M_1) \leq Z(M_1) \simeq p^2 \mathbb{Z} [\frac{1}{det(A)}]$ in additive notation. Then $f(g) = \frac{g}{p^2}$  and $p$ does not divide $det(A)$ implies that if $g \not= 0$ ( in additive notation) then for some sufficiently big $s \geq 1$ we have $f^s(g) \not\in Z(M_1)$ a contradiction. Thus $K \cap Z(M_1)$ is the trivial group.
 
 Note that $f$ induces a map $\widetilde{f} : W =\frac{M_1}{Z(M_1)} \to \frac{M}{Z(M)} = V$ that is nothing else than division by $p$ i.e. multiplication by $1/p$. Recall that
 $$V =  \cup_{i \geq 1} A^{-i} \mathbb{Z}^2  \subseteq  \cup_{i \geq 1} (\frac{1}{det(A)})^{i} \mathbb{Z}^2    \hbox{ and } W =    \cup_{i \geq 1} A^{-i} (p\mathbb{Z})^2 $$
 Since $det(A)$ is not divisible by $p$ we conclude that the unique element $w$  of $W$ such that $f^s(w)$ is defined for all $s \geq 1$ is the trivial one i.e. in additive notation $w = 0$. Thus  $K \cap M \subseteq K \cap Z(M_1)$ is the trivial group. 
 
 Finally since $K \cap M = 1$ and both $K$ and $M$ are normal in $G$ we conclude that $[K, M] = 1$. Furthermore $K \simeq KM/ M \leq G/ M \simeq \mathbb{Z}$, so $K$ is cyclic. Thus $M \times K $ is a subgroup of finite index in $G$, say index $s$. Hence $\varphi^s$ is the identity, hence $det(A)^s =1$ and $det(A) \in \mathbb{Z}$ implies $det(A) = \pm 1$, a contradiction. This completes the proof of the Claim 3.


\begin{thebibliography}{99}

\bibitem{Ba-S}  L. Bartholdi and  Z. Sunik,  Some solvable automaton groups
Contemp. Math., 394,
American Mathematical Society, Providence, RI, 2006, 11–29



\bibitem{B-S} G. Baumslag and D. Solitar, 
Some two-generator one-relator non-Hopfian groups.
Bull. Amer. Math. Soc. 68 (1962), 199–201

\bibitem{But} J. Button, Generalised Baumslag-Solitar groups and hierarchically hyperbolic groups.
arXiv preprint arXiv:2208.12688, 2022




\bibitem{B-I-W}   
K-U. Bux, C. Llosa Isenrich, and X. Wu,
On the Boone--Higman Conjecture for groups acting on locally finite trees, arXiv:2408.05673


\bibitem{Brazil} A. C. Dantas, T. M. G. Santos, and S. N. Sidki,
Intransitive self-similar groups. J. Algebra 567(2021), 564–581


\bibitem{Brazil2} A. C. Dantas and S. Sidki, On self-similarity of wreath products of abelian groups. Groups Geom. Dyn. 12 (2018), no. 3, 1061–1068  

\bibitem{D-S} C. Drutu and M. Sapir,   Non-linear residually finite groups,
J. Algebra 284 (2005), no. 1, 174–178  

\bibitem{G-M} 
Y. Glasner and S. Mozes, Automata and square complexes, Geom. Dedicata 111 (2005) 43–64 


\bibitem{For} M. Forester. On uniqueness of JSJ decompositions of finitely generated groups. Comment.
Math. Helv., 78(4):740–751, 2003



\bibitem{K-Si} D. H. Kochloukova and S. N.  Sidki, Self-similar groups of type FPn. Geom. Dedicata 204 (2020), 241 - 264

 \bibitem{K-S} D. Kochloukova and M. Souza, Non self-similar metabelian groups, preprint
 

 
 
 
 
  \bibitem{L-S}  J. Lopez de Gamiz Zearra and S. Shepherd,  Separability properties of higher rank GBS groups,
Bull. Lond. Math. Soc. 57 (2025), no. 4, 1171 - 1194



\bibitem{Lev1} G. Levitt,   Generalized Baumslag-Solitar groups: rank and finite index subgroups,
Ann. Inst. Fourier (Grenoble) 65 (2015), no. 2, 725–762

\bibitem{Lev2} G. Levitt,  Quotients and subgroups of Baumslag-Solitar groups,
J. Group Theory 18 (2015), no. 1, 1–43




\bibitem{Sidki} N. Gupta and S. Sidki,  On the Burnside problem for periodic groups. Math. Z. 182 (1983), no. 3, 385–388

\bibitem{G}  R. I. Grigorchuk, On Burnside's problem on periodic groups. (Russian) Funktsionalyi Analiz i ego Prilozheniya, vol. 14 (1980), no. 1, pp. 53 - 54

\bibitem{Hall} P. Hall,  On the finiteness of certain soluble groups,
Proc. London Math. Soc. (3) 9 (1959), 595 - 622 



\bibitem{H-W} T. Hsu and D. Wise,
Ascending HNN extensions of polycyclic groups are residually finite,
J. Pure Appl. Algebra 182 (2003), no. 1, 65 - 78


\bibitem{L-M} 
 I. J. Leary and A. Minasyan,  Commensurating HNN extensions: nonpositive curvature and biautomaticity,
Geom. Topol. 25 (2021), no. 4, 1819 - 1860

\bibitem{M} D. I. Moldavanskii,
Residual finiteness of descending HNN-extensions of groups,
Ukraïn. Mat. Zh. 44 (1992), no. 6, 842–845; translation in
Ukrainian Math. J. 44 (1992), no. 6, 758–760

\bibitem{Mal} A. Malcev, 
On isomorphic matrix representations of infinite groups,
Rec. Math. [Mat. Sbornik] N.S. 8/50 (1940), 405–422





\bibitem{M-R-V2} E. Raptis, O. Talelli, and D. Varsos,   On finiteness conditions of certain graphs of groups
, Internat. J. Algebra Comput. 5 (1995), no. 6, 719–724

\bibitem{N-S} 
V. Nekrashevych and S. Sidki, Automorphisms of the binary tree: state-closed subgroups and dynamics of 1/2-endomorphisms. Groups: topological, combinatorial and arithmetic aspects, 375 - 404, London Math. Soc. Lecture Note Ser., 311, Cambridge Univ. Press, Cambridge, 2004

\bibitem{book} V. Nekrashevych. Self-similar groups, volume
117 of Mathematical Surveys and Monographs. American
Mathematical Society, Providence, RI, 2005.

\bibitem{O} M. Olivier,  Which Nilpotent Groups are Self-Similar?, arXiv:2101.11291 

\bibitem{R-T-V} E. Raptis, O. Talelli, and D. Varsos,  On residual finiteness of graphs of nilpotent groups,
Internat. J. Algebra Comput. 14 (2004), no. 4, 403 - 408

\bibitem{Ros} J. E. Roseblade,  Group rings of polycyclic groups, J. Pure Appl. Algebra 3 (1973), 307 - 328


\bibitem{Sa} T. M. G. Santos,  Grupos autossimilares intransitivos, PhD Thesis, University of Brasilia, 2021

\bibitem{S-V} D. Savchuk and Y. Vorobets,  Automata generating free products of groups of order 2
J. Algebra 336 (2011), 53–66.




\bibitem{SVV}  
B. Steinberg, M. Vorobets, and Y. Vorobets,  Automata over a binary alphabet generating free groups of even rank
Internat. J. Algebra Comput. 21 (2011), no. 1-2, 329–354



\bibitem{V-V} M. Vorobets and Y. Vorobets, On a series of finite automata defining free transformation groups, Groups Geom. Dyn. 4 (2)
(2010) 377–405,

\bibitem{Why} K. Whyte, The large scale geometry of the higher Baumslag-Solitar groups. Geom. Funct.
Anal., 11(6):1327–1343, 2001
\end{thebibliography}
 \end{document}